\documentclass[a4paper,reqno]{amsart}
\usepackage[T1]{fontenc}
\usepackage[utf8]{inputenc}
\usepackage{lmodern}
\usepackage{amsmath,amssymb,mathtools}
\usepackage{microtype}
\usepackage[hidelinks,unicode]{hyperref}
\hypersetup{pdftitle={The Guo-Zhang-Qi conjecture and positive Laplace kernels for a digamma power family},pdfauthor={Valmir Krasniqi},pdfsubject={Digamma function; logarithmic complete monotonicity; positive Laplace kernels}}
\numberwithin{equation}{section}
\newtheorem{theorem}{Theorem}[section]
\newtheorem{lemma}[theorem]{Lemma}
\newtheorem{proposition}[theorem]{Proposition}
\newtheorem{corollary}[theorem]{Corollary}
\theoremstyle{remark}
\newtheorem{remark}[theorem]{Remark}
\newcommand{\Lap}{\mathcal{L}}
\newcommand{\dd}{\,\mathrm{d}}
\newcommand{\R}{\mathbb{R}}

\newcommand{\alzero}{\alpha_0}
\newcommand{\alstar}{\alpha_*}
\newcommand{\del}{\delta_0}

\newcommand{\Ph}{\Phi_\alpha}

\allowdisplaybreaks[1]
\title[The Guo--Zhang--Qi conjecture]{The Guo--Zhang--Qi conjecture and positive Laplace kernels for a digamma power family}
\author{Valmir Krasniqi}
\address{Institute of Science, Technology, Engineering and Mathematics (STEM),
Rr. Nekibe Kelmendi Nr. 26, 10000 Prishtinë, Republic of Kosovo}
\email{valmir@stem-ks.org}
\date{}
\begin{document}
\begin{abstract}
We study logarithmic complete monotonicity of a power family built from the digamma function.
An explicit Laplace representation of its negative logarithmic derivative reduces the question to positivity of a continuous kernel.
We prove a strict kernel bound on an explicit parameter range and obtain uniform lower bounds for all alternating logarithmic derivatives.
In particular, this gives an independent proof of the property conjectured by Guo, Zhang, and Qi on the entire positive half-line.
We also characterize the exact admissible parameter threshold by a variational minimum and show that the explicit sufficient bound is not optimal.
The argument combines monotonicity of an auxiliary digamma kernel with a regularized logarithmic convolution.
An appendix identifies a false auxiliary inequality in an earlier proof, while the present argument independently establishes its stated conclusion.
\end{abstract}
\maketitle
\enlargethispage{4pt}

\section{Introduction and the main result}

A function $g\in C^\infty(0,\infty)$ is called \emph{completely monotone} if
\begin{equation}\label{eq:cm}
 (-1)^n g^{(n)}(t)\ge0,\qquad t>0,\quad n=0,1,2,\ldots.
\end{equation}
A positive function $f\in C^\infty(0,\infty)$ is called \emph{logarithmically completely monotone} if
\begin{equation}\label{eq:lcm}
 (-1)^m(\log f)^{(m)}(t)\ge0,\qquad t>0,\quad m\ge1.
\end{equation}
We use the word \emph{strictly} when all the corresponding inequalities are strict.
The property in \eqref{eq:lcm} is equivalent to complete monotonicity of $-(\log f)'$.
The Bernstein--Widder theorem characterizes completely monotone functions as Laplace transforms of positive measures; see \cite{SSV,Widder}.
For the functions considered here, we construct the representing density directly.

Throughout the paper,
\begin{equation}\label{eq:theta}
 \Theta(t)=t\bigl(\log t-\psi(t)\bigr),\qquad
 f_\alpha(t)=\exp\bigl(-[\Theta(t)+\alpha]\log t\bigr),
\end{equation}
where $t>0$, $\alpha\in\R$, $\psi=\Gamma'/\Gamma$, and $\gamma=-\psi(1)$.
Alzer \cite{Alzer} proved that $t^\beta(\log t-\psi(t))$ is completely monotone on $(0,\infty)$ precisely when $\beta\le1$.
In particular, $\Theta$ belongs to this class.

The question arose from the gamma-function inequalities of Guo, Zhang, and Qi, who conjectured in \cite[Remark~8]{GZQ} logarithmic complete monotonicity on $(0,\infty)$ of
\begin{equation}\label{eq:q}
 q(t)=f_\gamma(t)=t^{\,t(\psi(t)-\log t)-\gamma}.
\end{equation}
Krasniqi and Shabani \cite{KS2014} subsequently claimed a proof on the whole positive half-line.
Guo and Qi \cite{GQ2015} identified the interval restriction in that product argument and proved the assertion on $(0,1)$.
The corrected arXiv version of Krasniqi and Shabani \cite{KS2016} likewise states and proves the result on $(0,1)$.
The reason for this restriction is visible in the factorization
\begin{equation}\label{eq:oldfactor}
 \log q(t)=\bigl(\Theta(t)+\gamma\bigr)(-\log t),
\end{equation}
because $-\log t$ is completely monotone on $(0,1)$, but not on $(0,\infty)$.
Closure under products therefore does not establish the assertion on the full half-line by this route.
Guo and Qi \cite[Section~4]{GQ2015} also considered the parameter family \eqref{eq:theta}, derived the necessary condition $\alpha\ge-1/2$, and posed determination of the best admissible parameter as an open problem.

Bouali \cite[Proposition~1.8]{Bouali} later stated the full-half-line property for $f_\alpha$ whenever $\alpha\ge-1/4$.
That assertion includes the original conjecture and is an essential part of its history.
However, an auxiliary inequality invoked in the argument surrounding \cite[Eq.~(1.11), p.~10]{Bouali} is false, as the analytic counterexample in Appendix~\ref{sec:appendix} shows.
This invalidates that step, without disproving Bouali's stated conclusion.
The present Laplace-kernel proof establishes that conclusion independently and extends it to a larger explicit parameter range.

Set
\begin{equation}\label{eq:constants}
 \del=\dfrac{e^{-\gamma-1}}{12},\qquad
 \alzero=-\dfrac12+\del
       =-0.482787549912\ldots.
\end{equation}
The decimal is included only for orientation; all proofs use the exact expression.

\begin{theorem}\label{thm:main}
For every $\alpha\ge\alzero$, the function $f_\alpha$ is strictly logarithmically completely monotone on $(0,\infty)$.
More precisely, for every integer $m\ge1$ and every $t>0$,
\begin{equation}\label{eq:mainbound}
 (-1)^m(\log f_\alpha)^{(m)}(t)
 > (\alpha-\alzero)\dfrac{(m-1)!}{t^m}\ge0.
\end{equation}
The strict inequality remains valid at $\alpha=\alzero$.
\end{theorem}

\begin{corollary}\label{cor:q}
The function $q$ in \eqref{eq:q} is strictly logarithmically completely monotone on $(0,\infty)$, and
\begin{equation}\label{eq:qbound}
 (-1)^m(\log q)^{(m)}(t)
 >\left(\gamma+\dfrac12-\dfrac{e^{-\gamma-1}}{12}\right)
    \dfrac{(m-1)!}{t^m},\qquad m\ge1,\quad t>0.
\end{equation}
\end{corollary}

The main device is an explicit representation
\begin{equation}\label{eq:reprintro}
 - (\log f_\alpha)'(t)
   =\int_0^\infty e^{-tu}K_\alpha(u)\dd u,
\end{equation}
in which the parameter enters only as an additive constant in $K_\alpha$.
Section~2 establishes the required digamma-kernel identities.
Section~3 justifies the logarithmic convolution and derives \eqref{eq:reprintro} for every real $\alpha$.
Section~4 proves the lower bound and endpoint estimates for $K_\alpha$.
Section~5 completes the proof, derives consequences, and characterizes the exact admissible threshold variationally.
The sufficient bound $\alzero$ is not claimed to be optimal.

\section{The auxiliary digamma kernel}

Define
\begin{equation}\label{eq:h}
 h(u)=\dfrac1{u^2}-\dfrac1{4\sinh^2(u/2)},\qquad u>0.
\end{equation}
The following representation, also used in \cite[Eq.~(1.7)]{Bouali}, is the starting point of the proof.
We include its derivation and the elementary kernel estimates for completeness.

\begin{lemma}\label{lem:repr}
For every $t>0$,
\begin{equation}\label{eq:thetaH}
 \Theta(t)=\dfrac12+H(t),\qquad
 H(t)=\int_0^\infty h(u)e^{-tu}\dd u.
\end{equation}
Moreover,
\begin{equation}\label{eq:mass}
 \int_0^\infty h(u)\dd u=\dfrac12.
\end{equation}
\end{lemma}

\begin{proof}
The standard digamma representation \cite[Eq.~(5.9.13)]{DLMF} gives
\begin{equation}\label{eq:digammaint}
 \log t-\psi(t)=\int_0^\infty e^{-tu}F(u)\dd u,\qquad
 F(u)=\dfrac1{1-e^{-u}}-\dfrac1u.
\end{equation}
At the origin,
\begin{equation}\label{eq:Fexp}
 F(u)=\dfrac12+\dfrac u{12}-\dfrac{u^3}{720}
          +\dfrac{u^5}{30240}+O(u^7).
\end{equation}
Hence $F(0+)=1/2$, while $F(u)\to1$ as $u\to\infty$.
Direct differentiation yields
\begin{equation}\label{eq:Fprime}
 F'(u)=\dfrac1{u^2}-\dfrac{e^{-u}}{(1-e^{-u})^2}=h(u).
\end{equation}
The function $F$ is bounded on the positive half-line.
Multiplying \eqref{eq:digammaint} by $t$ and integrating by parts therefore gives
\begin{align*}
 \Theta(t)
 &=\left[-e^{-tu}F(u)\right]_{u=0}^{u=\infty}
       +\int_0^\infty e^{-tu}F'(u)\dd u\\
 &=\dfrac12+\int_0^\infty e^{-tu}h(u)\dd u.
\end{align*}
Finally, integrating \eqref{eq:Fprime} over a finite interval and passing to the endpoints gives \eqref{eq:mass}.
Convergence follows as well from the expansions in the next lemma.
\end{proof}

\begin{lemma}\label{lem:h}
The function $h$ extends smoothly to $u=0$, is positive and strictly decreasing on $(0,\infty)$, and satisfies
\begin{equation}\label{eq:hbound}
 0<h(u)<\dfrac1{12},\qquad u>0.
\end{equation}
Its endpoint expansions are
\begin{align}
 h(u)&=\dfrac1{12}-\dfrac{u^2}{240}+\dfrac{u^4}{6048}+O(u^6)
      &&(u\to0+),\label{eq:hzero}\\
 h(u)&=\dfrac1{u^2}+O(e^{-u}),\qquad
 h'(u)=-\dfrac2{u^3}+O(e^{-u})
      &&(u\to\infty).\label{eq:hinfty}
\end{align}
\end{lemma}

\begin{proof}
The inequality $2\sinh(u/2)>u$ gives $h(u)>0$.
Expansion of the analytic expression at the removable singularity $u=0$ gives \eqref{eq:hzero} and $h(0)=1/12$.
Differentiating \eqref{eq:h} gives
\begin{equation}\label{eq:hprime}
 h'(u)=-\dfrac2{u^3}+\dfrac{\cosh(u/2)}{4\sinh^3(u/2)}.
\end{equation}
On writing $z=u/2$, the inequality $h'(u)<0$ is equivalent to
\begin{equation}\label{eq:sinhineq}
 \sinh^3 z>z^3\cosh z,\qquad z>0.
\end{equation}
To prove it, set
\[
 R_0(z)=3\log\left(\dfrac{\sinh z}{z}\right)-\log(\cosh z).
\]
Then $R_0(0+)=0$ and
\begin{equation}\label{eq:Rprime}
 R_0'(z)=\dfrac{D(z)}{z\sinh z\cosh z},\qquad
 D(z)=z(3+2\sinh^2 z)-3\sinh z\cosh z.
\end{equation}
A further differentiation gives
\[
 D'(z)=4\sinh z\,(z\cosh z-\sinh z).
\]
The function $E(z)=z\cosh z-\sinh z$ satisfies $E(0)=0$ and $E'(z)=z\sinh z>0$.
Consequently $D'(z)>0$, $D(z)>0$, and $R_0'(z)>0$ for $z>0$.
Thus $R_0(z)>0$, which proves \eqref{eq:sinhineq} and the strict decrease of $h$.
Together with $h(0)=1/12$, this proves \eqref{eq:hbound}.
Finally, the two assertions in \eqref{eq:hinfty} follow separately from the explicit formulas \eqref{eq:h} and \eqref{eq:hprime}; no differentiation of an unspecified remainder is required.
\end{proof}

In particular, the function
\begin{equation}\label{eq:b}
 b(u)=uh(u),\qquad b(0)=0,
\end{equation}
is continuously differentiable on $[0,\infty)$, and both $b$ and $b'$ are bounded.
Near zero, $b(u)=u/12+O(u^3)$; at infinity, $b(u)=O(u^{-1})$ and $b'(u)=O(u^{-2})$.
These estimates justify the convolution manipulations below.

\section{Construction of the Laplace kernel}

For a function $a$ with absolutely convergent transform, write
\[
 \Lap[a](t)=\int_0^\infty e^{-tu}a(u)\dd u,\qquad t>0.
\]
By \eqref{eq:thetaH},
\begin{equation}\label{eq:phi}
 \Ph(t):=-(\log f_\alpha)'(t)
  =\dfrac{\alpha+1/2}{t}+\dfrac{H(t)}t+H'(t)\log t.
\end{equation}
Let $B(t)=\Lap[b](t)$.
Differentiation under the integral defining $H$, justified by Lemma~\ref{lem:h}, gives $H'(t)=-B(t)$.
The only term in \eqref{eq:phi} requiring special treatment is $H'(t)\log t$.

\subsection{A regularized logarithmic convolution}
Set $\ell(u)=\log u$ for $u>0$ and define
\begin{equation}\label{eq:w}
 w(u)=(b*\ell)(u)=\int_0^u b(v)\log(u-v)\dd v
                   =\int_0^u b(u-s)\log s\dd s.
\end{equation}

\begin{lemma}\label{lem:convolution}
The function $w$ extends continuously to $0$ with $w(0)=0$, is continuously differentiable on $(0,\infty)$, and
\begin{equation}\label{eq:wprime}
 w'(u)=b(u)\log u+
       \int_0^u\dfrac{b(v)-b(u)}{u-v}\dd v.
\end{equation}
The integral has a removable endpoint singularity.
For every $t>0$, the following identities are valid with absolute convergence:
\begin{equation}\label{eq:convtransforms}
 \Lap[w](t)=B(t)\Lap[\ell](t),\qquad
 \Lap[w'](t)=t\Lap[w](t).
\end{equation}
\end{lemma}

\begin{proof}
Since $b(0)=0$ and $b'$ is bounded, the second form of \eqref{eq:w} can be differentiated on compact subintervals of $(0,\infty)$, giving
\begin{equation}\label{eq:wprimealt}
 w'(u)=\int_0^u b'(u-s)\log s\dd s
       =\int_0^u b'(v)\log(u-v)\dd v.
\end{equation}
For clarity, one may first restrict to $s\ge\varepsilon$.
The omitted integrals are bounded uniformly on compact $u$-intervals by a constant times $\int_0^\varepsilon|\log s|\dd s$, which tends to zero.
The upper-end contribution in the Leibniz rule vanishes because $b(0)=0$.

For $0<\varepsilon<u$, integration by parts yields
\begin{align*}
 \int_0^{u-\varepsilon}b'(v)\log(u-v)\dd v
 &=\bigl(b(u-\varepsilon)-b(u)\bigr)\log\varepsilon
    +b(u)\log u\\
 &\quad+\int_0^{u-\varepsilon}\dfrac{b(v)-b(u)}{u-v}\dd v.
\end{align*}
The first term tends to zero, since it is $O(\varepsilon|\log\varepsilon|)$.
The quotient in the last integral tends to $-b'(u)$ as $v\to u-$.
Passing to the limit proves \eqref{eq:wprime}.

Near the origin, $b(u)=O(u)$ and \eqref{eq:w} imply
\begin{equation}\label{eq:wzero}
 w(u)=O\bigl(u^2(1+|\log u|)\bigr),\qquad
 w'(u)=O\bigl(u(1+|\log u|)\bigr).
\end{equation}
The first estimate gives $w(0)=0$.
The boundedness of $b$ and $b'$ also gives, as $u\to\infty$,
\begin{equation}\label{eq:wlarge}
 |w(u)|+|w'(u)|=O\bigl(u(1+\log u)\bigr).
\end{equation}
Thus $w,w'$ have absolutely convergent Laplace transforms and $e^{-tu}w(u)\to0$ at infinity for every $t>0$.

The use of the convolution theorem is justified explicitly by
\begin{align*}
 &\int_0^\infty\int_0^\infty
 e^{-t(v+s)}|b(v)\log s|\dd s\dd v\\
 &\hspace{15mm}=
 \left(\int_0^\infty e^{-tv}|b(v)|\dd v\right)
 \left(\int_0^\infty e^{-ts}|\log s|\dd s\right)<\infty.
\end{align*}
Fubini's theorem therefore proves the first identity in \eqref{eq:convtransforms}.
Integration by parts, using \eqref{eq:wzero} and \eqref{eq:wlarge}, proves the second.
\end{proof}

\subsection{The explicit density}
Differentiating the gamma integral
\[
 \int_0^\infty e^{-tu}u^{a-1}\dd u=\dfrac{\Gamma(a)}{t^a}
\]
with respect to $a$ at $a=1$ gives
\begin{equation}\label{eq:logtransform}
 G(t):=\Lap[\ell](t)=-\dfrac{\gamma+\log t}{t}.
\end{equation}
Differentiation is justified by an integrable majorant for $a$ in a compact neighborhood of $1$.
Since $H'=-B$ and $\log t=-tG(t)-\gamma$, Lemma~\ref{lem:convolution} yields
\begin{equation}\label{eq:cross}
 H'(t)\log t=tB(t)G(t)+\gamma B(t)
               =\Lap[w'+\gamma b](t).
\end{equation}

\begin{proposition}\label{prop:kernel}
For every $\alpha\in\R$ and every $t>0$,
\begin{equation}\label{eq:kernelrepr}
 \Ph(t)=\int_0^\infty e^{-tu}K_\alpha(u)\dd u,
\end{equation}
where
\begin{equation}\label{eq:kernel}
 K_\alpha(u)=\alpha+\dfrac12
       +uh(u)(\gamma+\log u)
       +u\int_0^u\dfrac{h(v)-h(u)}{u-v}\dd v.
\end{equation}
The integral in \eqref{eq:kernel} is an ordinary convergent integral, with its integrand extended continuously at $v=u$.
\end{proposition}

\begin{proof}
Tonelli's theorem gives
\begin{equation}\label{eq:Hbyt}
 \dfrac{H(t)}t
 =\Lap\left[u\longmapsto\int_0^u h(v)\dd v\right](t).
\end{equation}
Combining \eqref{eq:phi}, \eqref{eq:cross}, \eqref{eq:Hbyt}, and \eqref{eq:wprime}, the initial expression for the density is
\begin{align*}
 K_\alpha(u)
 &=\alpha+\dfrac12+b(u)(\gamma+\log u)\\
 &\quad+\int_0^u h(v)\dd v
       +\int_0^u\dfrac{b(v)-b(u)}{u-v}\dd v.
\end{align*}
Using $b(v)=vh(v)$, the two last integrands combine to
\[
 h(v)+\dfrac{vh(v)-uh(u)}{u-v}
     =u\dfrac{h(v)-h(u)}{u-v}.
\]
This proves \eqref{eq:kernel} and \eqref{eq:kernelrepr}.
The convergence assertions follow from Lemma~\ref{lem:convolution} and the local smoothness of $h$.
\end{proof}

\section{Positivity and endpoint behavior of the kernel}

Write
\begin{equation}\label{eq:J}
 J(u)=u\int_0^u\dfrac{h(v)-h(u)}{u-v}\dd v.
\end{equation}
By the strict decrease of $h$,
\begin{equation}\label{eq:Jpositive}
 J(u)>0,\qquad u>0.
\end{equation}
This term retains its sign independently of the parameter and independently of the sign of $\log u+\gamma$.

\begin{proposition}\label{prop:positive}
For every $\alpha\in\R$ and $u>0$,
\begin{equation}\label{eq:Klower}
 K_\alpha(u)>\alpha+\dfrac12-\del=\alpha-\alzero.
\end{equation}
In particular, $K_\alpha$ is strictly positive on $(0,\infty)$ whenever $\alpha\ge\alzero$.
\end{proposition}

\begin{proof}
If $u\ge e^{-\gamma}$, then $uh(u)(\gamma+\log u)\ge0$, and \eqref{eq:Jpositive} gives
\[
 K_\alpha(u)>\alpha+\dfrac12>\alpha+\dfrac12-\del.
\]
If $0<u<e^{-\gamma}$, multiplication of $h(u)<1/12$ by the negative number $u(\gamma+\log u)$ gives
\[
 uh(u)(\gamma+\log u)>\dfrac u{12}(\gamma+\log u).
\]
The function $u\mapsto u(\gamma+\log u)$ has its global minimum $-e^{-\gamma-1}$ at $u=e^{-\gamma-1}$.
Dropping the strictly positive term $J(u)$ therefore yields
\[
 K_\alpha(u)>\alpha+\dfrac12+\dfrac u{12}(\gamma+\log u)
              \ge\alpha+\dfrac12-\del.
\]
This proves \eqref{eq:Klower} in both cases.
\end{proof}

\begin{proposition}\label{prop:growth}
For every fixed $\alpha\in\R$, the kernel $K_\alpha$ is bounded and continuous on $(0,\infty)$, extends continuously to $0$, and satisfies
\begin{align}
 K_\alpha(u)
 &=\alpha+\dfrac12+\dfrac u{12}(\gamma+\log u)
       +O(u^3|\log u|), &&u\to0+,\label{eq:Kzero}\\
 \lim_{u\to\infty}K_\alpha(u)&=\alpha+1.\label{eq:Klimit}
\end{align}
Consequently, for every $t>0$ and every integer $n\ge0$,
\begin{equation}\label{eq:integrability}
 \int_0^\infty u^n e^{-tu}|K_\alpha(u)|\dd u<\infty.
\end{equation}
\end{proposition}

\begin{proof}
Continuity follows from the removable singularity in the integrand of \eqref{eq:J}; on compact intervals, its continuous extension is uniformly bounded.
Near zero, the mean value theorem and \eqref{eq:hzero} imply
\[
 \sup_{0\le v<u}\left|\dfrac{h(v)-h(u)}{u-v}\right|=O(u),
 \qquad J(u)=O(u^3).
\]
Combining this with $uh(u)=u/12+O(u^3)$ proves \eqref{eq:Kzero}.

For large $u$, split $J(u)=J_1(u)+J_2(u)$ at $v=u/2$.
On the first interval,
\begin{equation}\label{eq:J1}
 J_1(u)=\int_0^{u/2}\dfrac{u}{u-v}h(v)\dd v-uh(u)\log2.
\end{equation}
In the first integral, the factor $u/(u-v)$ is at most $2$ and tends to $1$ for each fixed $v$.
Since $h$ is integrable, dominated convergence and \eqref{eq:mass} show that this integral tends to $1/2$.
The second term tends to zero by \eqref{eq:hinfty}.
Thus $J_1(u)\to1/2$.
On the remaining interval, the mean value theorem gives
\begin{equation}\label{eq:J2}
 0\le J_2(u)
 \le\dfrac{u^2}{2}\sup_{u/2\le s\le u}|h'(s)|=O(u^{-1}).
\end{equation}
Therefore $J(u)\to1/2$.
Also,
\[
 uh(u)(\gamma+\log u)=O\left(\dfrac{\log u}{u}\right)\longrightarrow0.
\]
Substituting in \eqref{eq:kernel} proves \eqref{eq:Klimit}.
Continuity and the finite limits at both endpoints give boundedness, and the exponential factor then proves \eqref{eq:integrability}.
\end{proof}

\section{Proof of the main result and consequences}

\subsection{Alternating logarithmic derivatives}
The representation for all real parameters will also be useful in discussing necessity.

\begin{proposition}\label{prop:allorders}
For every $\alpha\in\R$, every integer $m\ge1$, and every $t>0$,
\begin{equation}\label{eq:allorders}
 (-1)^m(\log f_\alpha)^{(m)}(t)
   =\int_0^\infty u^{m-1}e^{-tu}K_\alpha(u)\dd u.
\end{equation}
\end{proposition}

\begin{proof}
On a compact interval $t\in[a,b]\subset(0,\infty)$, the integrand after $n$ differentiations in \eqref{eq:kernelrepr} is bounded in absolute value by
\[
 u^n e^{-au}|K_\alpha(u)|,
\]
which is integrable by Proposition~\ref{prop:growth}.
Differentiation under the integral is therefore justified at every order.
Since $\Ph=-(\log f_\alpha)'$, setting $n=m-1$ gives \eqref{eq:allorders}.
\end{proof}

\begin{proof}[Proof of Theorem~\ref{thm:main} and Corollary~\ref{cor:q}]
For $\alpha\ge\alzero$, Propositions~\ref{prop:positive} and \ref{prop:allorders} give
\begin{align*}
 (-1)^m(\log f_\alpha)^{(m)}(t)
 &>(\alpha-\alzero)\int_0^\infty u^{m-1}e^{-tu}\dd u\\
 &=(\alpha-\alzero)\dfrac{(m-1)!}{t^m}.
\end{align*}
The difference between the two sides is the integral of a strictly positive continuous function on $(0,\infty)$, so the inequality is strict even when $\alpha=\alzero$.
This proves Theorem~\ref{thm:main}.
Taking $\alpha=\gamma$ proves Corollary~\ref{cor:q}.
\end{proof}

\subsection{Monotonicity, positive powers, and endpoint asymptotics}

\begin{corollary}\label{cor:consequences}
For every $\alpha\ge\alzero$, the function $f_\alpha$ is strictly decreasing and strictly log-convex on $(0,\infty)$.
Every positive power $f_\alpha^r$, $r>0$, is strictly completely monotone.
Moreover,
\begin{equation}\label{eq:asymptoticf}
 f_\alpha(t)\sim t^{-(\alpha+1)}\quad(t\to0+),\qquad
 f_\alpha(t)\sim t^{-(\alpha+1/2)}\quad(t\to\infty).
\end{equation}
In particular, $f_\alpha(1)=1$ and $f_\alpha$ is a bijection of $(0,\infty)$ onto $(0,\infty)$.
\end{corollary}

\begin{proof}
The first assertion follows from the cases $m=1,2$ of Theorem~\ref{thm:main}.
For the second, let $F=f_\alpha^r$ and $A=-(\log F)'=r\Ph$.
The function $A$ is strictly completely monotone, and $F'=-AF$.
Leibniz' rule gives
\begin{equation}\label{eq:powers}
 (-1)^{n+1}F^{(n+1)}
 =\sum_{k=0}^n\binom nk
     \bigl((-1)^kA^{(k)}\bigr)
     \bigl((-1)^{n-k}F^{(n-k)}\bigr).
\end{equation}
Starting with $F>0$, induction proves strict complete monotonicity of $F$.

The standard expansions \cite[Eqs.~(5.7.6) and (5.11.2)]{DLMF} are
\begin{align*}
 \psi(t)&=-\dfrac1t-\gamma+O(t), &&t\to0+,\\
 \psi(t)&=\log t-\dfrac1{2t}-\dfrac1{12t^2}+O(t^{-4}),
       &&t\to\infty.
\end{align*}
Thus
\[
 \Theta(t)=1+t\log t+\gamma t+O(t^2)\quad(t\to0+),
\]
and
\[
 \Theta(t)=\dfrac12+\dfrac1{12t}+O(t^{-3})\quad(t\to\infty).
\]
After multiplication by $\log t$, these identities show that
\begin{align*}
 \log\bigl(t^{\alpha+1}f_\alpha(t)\bigr)&\longrightarrow0
     &&(t\to0+),\\
 \log\bigl(t^{\alpha+1/2}f_\alpha(t)\bigr)&\longrightarrow0
     &&(t\to\infty),
\end{align*}
which proves \eqref{eq:asymptoticf}.
Since $\alpha\ge\alzero>-1/2$, the endpoint limits of $f_\alpha$ are respectively $\infty$ and $0$.
The remaining assertions follow from continuity, strict decrease, and the definition at $t=1$.
\end{proof}

\subsection{A variational characterization of the exact threshold}
The explicit sufficient constant $\alzero$ comes from dropping the positive term $J$ and replacing $h$ by $1/12$ where the logarithmic factor is negative.
The full kernel gives an exact, but implicit, description of the admissible parameters.
Define
\begin{equation}\label{eq:rho}
 \rho(u)=\dfrac12+uh(u)(\gamma+\log u)+J(u),\qquad u>0,
\end{equation}
so that $K_\alpha(u)=\alpha+\rho(u)$.

\begin{proposition}\label{prop:sharp}
The function $\rho$ attains its global minimum on $(0,\infty)$.
Set
\begin{equation}\label{eq:alphastar}
 \alstar=-\min_{u>0}\rho(u).
\end{equation}
Then
\begin{equation}\label{eq:thresholdbounds}
 -\dfrac12<\alstar<\alzero,
\end{equation}
and
\begin{equation}\label{eq:iff}
 f_\alpha\text{ is logarithmically completely monotone on }(0,\infty)
 \quad\Longleftrightarrow\quad \alpha\ge\alstar.
\end{equation}
For every admissible parameter, including $\alpha=\alstar$, the logarithmic complete monotonicity is strict.
\end{proposition}

\begin{proof}
By Proposition~\ref{prop:growth},
\[
 \rho(0+)=\dfrac12,\qquad \rho(\infty)=1.
\]
Also, \eqref{eq:Kzero} gives
\[
 \rho(u)-\dfrac12
   =\dfrac u{12}(\gamma+\log u)+O(u^3|\log u|)<0
\]
for all sufficiently small positive $u$.
Consequently the infimum of $\rho$ is strictly below both endpoint limits, so continuity implies that it is attained at some $u_0\in(0,\infty)$.
Proposition~\ref{prop:positive} with $\alpha=0$ gives
\[
 \rho(u_0)>\dfrac12-\del=-\alzero,
\]
whereas $\rho(u_0)<1/2$.
Negating these inequalities proves \eqref{eq:thresholdbounds}.

If $\alpha\ge\alstar$, then $K_\alpha(u)\ge0$ for every $u>0$.
Since
\[
 K_\alpha(0)=\alpha+\dfrac12\ge\alstar+\dfrac12>0,
\]
the kernel is strictly positive on a nonempty interval adjacent to zero.
Equation~\eqref{eq:allorders} therefore proves strict logarithmic complete monotonicity.

For necessity, we give an elementary localization argument.
Suppose $f_\alpha$ is logarithmically completely monotone, and fix $u_0>0$.
For $n\ge1$, let
\begin{equation}\label{eq:pn}
 p_n(u)=\dfrac{(n/u_0)^{n+1}}{n!}\,u^n e^{-nu/u_0},\qquad u>0.
\end{equation}
These nonnegative kernels have mass one, and direct integration gives
\begin{equation}\label{eq:moments}
 \int_0^\infty (u-u_0)^2p_n(u)\dd u
        =\dfrac{u_0^2(n+2)}{n^2}\longrightarrow0.
\end{equation}
Because $K_\alpha$ is bounded and continuous, splitting its integral into $|u-u_0|<\varepsilon$ and its complement, and using \eqref{eq:moments}, proves
\[
 \int_0^\infty p_n(u)K_\alpha(u)\dd u\longrightarrow K_\alpha(u_0).
\]
On the other hand, \eqref{eq:allorders} with $m=n+1$ and $t=n/u_0$ shows that each integral on the left is nonnegative.
Thus $K_\alpha(u_0)\ge0$.
Since $u_0$ was arbitrary, $\alpha+\rho(u)\ge0$ for all $u>0$, or equivalently $\alpha\ge\alstar$.
\end{proof}

\begin{remark}\label{rem:sharp}
Proposition~\ref{prop:sharp} separates the explicit sufficient bound from the actual threshold.
It does not evaluate the minimum in \eqref{eq:alphastar} in closed form, and it does not assert uniqueness of its minimizer.
In particular, the necessary condition $\alpha\ge-1/2$ discussed in \cite[Section~4]{GQ2015} is not sufficient.
There is also a direct check at the endpoint: since $H(t)>0$,
\[
 f_{-1/2}(t)=\exp\bigl(-H(t)\log t\bigr)<1\qquad(t>1),
\]
while \eqref{eq:asymptoticf}, whose asymptotic derivation holds for every fixed real $\alpha$, gives $f_{-1/2}(t)\to1$ as $t\to\infty$.
Such a function cannot be decreasing on $(1,\infty)$.
\end{remark}

\begin{remark}
The construction shows why an additive parameter is effective: changing $\alpha$ translates the entire Laplace density vertically.
The estimates of Section~4 yield an explicit admissible interval without having to locate its minimum.
A closed-form evaluation of \eqref{eq:alphastar}, or sharper explicit estimates for it, would require a more precise analysis of the regularized integral in \eqref{eq:rho}.
No extension to discrete or basic digamma analogues is asserted here.
\end{remark}

\appendix
\section{An auxiliary inequality in an earlier argument}\label{sec:appendix}

This appendix records the precise obstruction in the proof of \cite[Proposition~1.8]{Bouali}, referring specifically to arXiv version~1, dated 3~June~2022.
The argument on page~10 asserts that
\begin{equation}\label{eq:gn}
 g_n(x):=\dfrac{n!}{4x^{n+1}}
             +(-1)^n\Theta^{(n+1)}(x)\log x\ge0
 \qquad(n\ge1,\ x\ge1).
\end{equation}
We show that $g_{20}(5)<0$ by an exact estimate rather than by a numerical sign test.

From \eqref{eq:thetaH},
\[
 \Theta^{(21)}(5)=-\int_0^\infty u^{21}e^{-5u}h(u)\dd u.
\]
Put
\begin{equation}\label{eq:A}
 A=\dfrac{5^{21}}{20!}\int_0^\infty u^{21}e^{-5u}h(u)\dd u.
\end{equation}
Then
\begin{equation}\label{eq:scaledg}
 \dfrac{5^{21}}{20!}g_{20}(5)=\dfrac14-A\log5.
\end{equation}
For $u>0$, the geometric series gives
\[
 \dfrac{e^{-u}}{(1-e^{-u})^2}=\sum_{k=1}^\infty ke^{-ku},
 \qquad
 h(u)=\dfrac1{u^2}-\sum_{k=1}^\infty ke^{-ku}.
\]
Both terms are integrable after multiplication by $u^{21}e^{-5u}$.
Tonelli's theorem for the nonnegative series and the gamma integral therefore yield
\begin{equation}\label{eq:Aseries}
 A=\dfrac14-21\cdot5^{21}
          \sum_{k=1}^\infty\dfrac{k}{(k+5)^{22}}.
\end{equation}
The function $s\mapsto s/(s+5)^{22}$ is strictly decreasing for $s\ge1$, since its derivative is $(5-21s)/(s+5)^{23}$.
Hence
\begin{align}
 \sum_{k=1}^\infty\dfrac{k}{(k+5)^{22}}
 &\le\dfrac1{6^{22}}+\dfrac2{7^{22}}
                 +\int_2^\infty\dfrac{s}{(s+5)^{22}}\dd s\notag\\
 &=\dfrac1{6^{22}}+\dfrac2{7^{22}}
       +\dfrac1{20\cdot7^{20}}-\dfrac5{21\cdot7^{21}}.
       \label{eq:seriesbound}
\end{align}
Let $T$ be $21\cdot5^{21}$ times the last expression.
This is a rational number, and direct simplification gives
\begin{equation}\label{eq:rationalcheck}
 \dfrac1{12}-T
 =\dfrac{3{,}096{,}318{,}751{,}192{,}320{,}520{,}033{,}784{,}171{,}707}
        {24{,}505{,}586{,}101{,}564{,}400{,}033{,}358{,}935{,}683{,}497{,}984}
 >0.
\end{equation}
Equations~\eqref{eq:Aseries}--\eqref{eq:rationalcheck} therefore give $A\ge1/4-T>1/6$.
Also $\log5>3/2$.
For an elementary verification of the latter inequality, the exponential series gives
\[
 e=\sum_{j=0}^\infty\dfrac1{j!}
 <\dfrac83+\dfrac1{24}\sum_{j=0}^\infty\dfrac1{5^j}
 =\dfrac{87}{32}<\dfrac{11}{4},
\]
so $e^3<(11/4)^3<25$ and hence $e^{3/2}<5$.
Substitution into \eqref{eq:scaledg} now yields
\[
 \dfrac{5^{21}}{20!}g_{20}(5)
 <\dfrac14-\dfrac16\cdot\dfrac32=0.
\]
This disproves \eqref{eq:gn}.
It does not disprove the logarithmic complete monotonicity asserted in \cite[Proposition~1.8]{Bouali}; that conclusion follows independently from Theorem~\ref{thm:main}, since $-1/4>\alzero$.

\enlargethispage{3\baselineskip}

\end{document}